\documentclass[a4paper,USenglish,cleveref,thm-restate]{lipics-v2021}
\pdfoutput=1

\usepackage{shortcuts/style}
\usepackage{shortcuts/shortcuts}

\title{Non-Additive Discrepancy: Coverage Functions in a Beck-Fiala Setting} 

\author{Tatiana Rocha Avila}{Goethe University Frankfurt, Germany}{rochaavila@em.uni-frankfurt.de}{https://orcid.org/0009-0008-6450-2444}{Funded by DFG Research Unit ADYN under grant DFG 411362735.}
\author{Lars Rohwedder}{University of Southern Denmark, Odense, Denmark}{rohwedder@sdu.dk}{https://orcid.org/0000-0002-9434-4589}{}
\author{Leo Wennmann}{University of Southern Denmark, Odense, Denmark}{wennmann@imada.sdu.dk}{https://orcid.org/0009-0001-3346-6494}{}

\authorrunning{T. R.~Avila, L.~Rohwedder and L.~Wennmann} 
\Copyright{Tatiana Rocha Avila and Lars Rohwedder and Leo Wennmann}



\relatedversion{}\relatedversiondetails{Full Version}{https://arxiv.org/abs/2602.09948}

\funding{\textit{Lars Rohwedder and Leo Wennmann}: Supported by Dutch Research Council (NWO) project "The Twilight Zone of Efficiency: Optimality of Quasi-Polynomial Time Algorithms" [grant number OCEN.W.21.268].}
\acknowledgements{We thank the anonymous reviewers for many helpful comments.}

\allowdisplaybreaks
\nolinenumbers

\EventEditors{Philip Bille, Seth Pettie, and Sabine Storandt}
\EventNoEds{3}
\EventLongTitle{34th Annual European Symposium on Algorithms (ESA 2026)}
\EventShortTitle{ESA 2026}
\EventAcronym{ESA}
\EventYear{2026}
\EventDate{August 31--September 4, 2026}
\EventLocation{L'Aquila, Italy}
\EventLogo{}
\SeriesVolume{388}
\ArticleNo{137}

\begin{document}

\maketitle

\begin{abstract}
Recent concurrent work by Dupr{\'{e}} la Tour and Fujii and by Hollender, Manurangsi, Meka, and Suksompong [ITCS'26] introduced a generalization of classical discrepancy theory to non-additive functions, motivated by applications in fair division.
As many classical techniques from discrepancy theory seem to fail in this setting, including linear algebraic methods like the Beck-Fiala Theorem [Discrete Appl. Math '81], it remains widely open whether comparable non-additive bounds can be achieved.

Towards a better understanding of non-additive discrepancy, we study coverage functions in a sparse setting comparable to the classical Beck-Fiala Theorem. 
Our setting generalizes the additive Beck-Fiala setting, rank functions of partition matroids, and edge coverage in graphs. 
More precisely, assuming each of the $n$ items covers only $t$ elements across all functions, we prove a constructive discrepancy bound that is polynomial in $t$, the number of colors~$k$, and $\log n$.
\end{abstract}

\section{Introduction}
The field of Discrepancy Theory studies
the inevitable imbalance that arises under discrete choices. 
In the classical setting, the vertices of a hypergraph
have to be colored in red or blue while bounding
the difference in the number of red and blue vertices
for each hyperedge. The maximum difference is called
the discrepancy of the coloring and the typical goal is to prove
existence of a coloring of some discrepancy, possibly under structural
restrictions on the hypergraph.
One of the cornerstones of this area is the classical result by Spencer~\cite{spencer1985six} which proves that any hypergraph with $n$
vertices and~$m = n$ hyperedges has a coloring of discrepancy~$\OO(\sqrt{n})$. 
Another classical result by Beck and Fiala~\cite{BeckF81} states
that any hypergraph where each vertex appears in at most $t$ hyperedges
has a coloring of discrepancy~$\OO(t)$.

Notably, several classical results in discrepancy theory were
initially non-constructive. This includes Spencer's Theorem~\cite{spencer1985six} as well as Banaszczyk's proof \cite{Banaszczyk98}
of a $\OO(\sqrt{\log n})$ bound in the Koml\'os
setting\footnote{The Koml\'os conjecture states that given
vectors of euclidean norm at most $1$, one can find signs
such that their signed sum is at most $\OO(1)$ in every dimension.}, 
which implies $\OO(\sqrt{t\log n})$ in the Beck-Fiala setting. 
As both remained longstanding open problems, it took significant effort until efficient algorithms were known~\cite{Bansal10,LovettM15,Rothvoss17,DadushGLN19,BansalDGL19}.
Remarkably, there is still progress on longstanding conjectures, with Bansal and Jiang~\cite{BansalJ25Improved,BansalJ25} very recently proving $\tilde \OO(\log^{1/4} n)$ for the Koml{\'{o}}s setting and~$\tilde \OO(\sqrt{t} + \sqrt{\log n})$ in the Beck-Fiala\footnote{The well known Beck-Fiala conjecture states that a discrepancy of $\OO(\sqrt{t})$ is possible.} setting.\footnote{Here, $\tilde \OO(\cdot)$ hides $\log\log n$ factors.}

Beyond the classical results, various extensions of the notion of discrepancy have been
studied. Due to the close connection of rounding linear programming relaxations,
techniques developed in this field have proven
useful in, for example, approximation algorithms~\cite{BansalRS22, HobergR17, RohwedderS25}.
Recently, concepts and techniques from discrepancy have been
studied in the context of fair allocation~\cite{ManurangsiS22,CaragiannisLS25,HollenderMMS26,DupreTF25}, with the following motivation:
consider the task of distributing~$n$ items among $k$ agents.
Suppose we have $m$ valuation functions $f_1,\dotsc,f_m : 2^n \rightarrow \mathbb R$, which each assign a value to every subset of items.
The different functions may represent incomparable
aspects of the items, for example monetary or sentimental value.
Our goal is to distribute the items among the agents such that each agent receives approximately the same value with respect to each of the functions.
Formally, we define:
\begin{definition}
    \label{def:multicolor-discrepancy}
    Let $\F = \{f_1, \dotsc , f_m: \set{0,1}^n \rightarrow \R \}$ be a family of functions. 
    The discrepancy of a $k$-coloring~$\chi : [n] \rightarrow [k]$ with respect to $\F$ is defined as
    \begin{equation*}
        \disc(\F, \chi) := \max_{i \in [m]} \max_{\ell,\ell' \in [k]} \big| f_i(\chi^{-1}(\ell)) - f_i(\chi^{-1}(\ell')) \big|.
    \end{equation*}
    Then, the $k$-color discrepancy of $\F$ is defined as
    \begin{equation*}
        \disc_\chi(\F, k) := \min_{\chi : [n] \rightarrow [k]} \disc(\F,\chi).
    \end{equation*}
\end{definition}
It is not hard to see that this notion of fairly allocating items
is a generalization of the classical setting of discrepancy, also referred to as multicolor discrepancy \cite{doerr2003multicolour}.
More precisely, consider one item for each vertex, $k=2$ agents
representing the two colors, and one function~$f$ for each hyperedge $H$
such that~$f(S) = |S\cap H|$.
Notably, many techniques from classical discrepancy
extend to more than $2$ colors and additive functions with bounded coefficients.
The setting of non-additive functions, however, is largely unexplored and the
transfer of techniques from classical discrepancy seems highly unclear. Partly, because many classical techniques
are linear algebraic in nature, \eg~the Beck-Fiala Theorem~\cite{BeckF81}, which intuitively leads to problems when the underlying functions are no longer linear.
Moreover, many additive results \cite{spencer1985six,Bansal10,Rothvoss17,LovettM15,BansalJM23} rely on the partial coloring method introduced by Beck~\cite{Beck81b} and Gluskin~\cite{gluskin1989extremal} which fails for general non-additive Lipschitz functions.

The non-additive setting was introduced simultaneously and independently by Hollender, Manurangsi, Meka and Suksompong~\cite{HollenderMMS26} and by Dupr{\'{e}} la Tour and Fujii~\cite{DupreTF25}.
Both show similar discrepancy bounds, but the analysis of the second set of authors yields the following slightly stronger result.
\begin{theorem}[\cite{DupreTF25}]
    \label{thm:discrepancy-for-primepower-k}
    For any prime power $k = p^{\nu}$ and any family of 1-Lipschitz functions $\F = \set{f_1, \dotsc , f_m: \set{0,1}^n \rightarrow \R}$, there exists a $k$-coloring of discrepancy $\OO(\sqrt{m \log (mk)})$.
\end{theorem}
To prove this, both sets of authors use the multilinear extension for a continuous relaxation of the problem and replace the typical linear algebraic method to obtain a solution with at most~$m$ fractional elements by a necklace cutting theorem \cite{alon1987splitting,JojicPZ21}.
Combining this with a standard randomized rounding approach and exploiting the Lipschitzness in the analysis yields the result.
Importantly, it is not constructive as the
necklace cutting theorem is non-constructive.

A bound on the Lipschitzness is necessary to obtain any reasonable
result, thus the setting of the theorem above
is perhaps the most general with respect to this assumption. However, the theorem still leaves a multitude 
of directions open, for example:
\begin{itemize}
    \item What about Beck-Fiala type settings?
    \item What about non-prime powers $k$?
    \item What about constructive bounds?
\end{itemize}
Towards these directions, we consider
a more restrictive class of functions, namely coverage functions,
which are defined in the next subsection.

\subsection{Sparse Coverage Functions}
We call a function $f: 2^n \rightarrow \mathbb Z_{\ge 0}$ a coverage function if there exist a universe~$U$ and subsets~$U_1,\dotsc,U_n \subseteq U$
such that $f(T) = |\bigcup_{i\in T} U_i|$ for all $T\subseteq [n]$.
We say that~$i$ covers~$u \in U$ if~$u\in U_i$. In other words,~$f(T)$ is the number of elements covered by the items~$T$.
Intuitively, coverage functions reward diversity over redundancy, because their function value grows with the number of covered    elements. Consequently, they are a natural subclass of submodular functions
that appear in classical problems of combinatorial optimization, for example the Set Cover problem \cite{ZorbasGKD10} or Max $k$-Coverage \cite{hochbaum}. For some other examples, refer to~\cite{Chakrabarty,DuLBS14, KarimiLH017,Cornuejols,feldman}.

We focus on the sparse setting where we are given coverage functions
$f_1,\dotsc,f_m$ and each item covers at most $t$ elements across all
$m$ functions. We say that such a family of functions is $t$-sparse (which we formally define in the next section).
Note that such functions are $t$-Lipschitz.
Examples for $t$-sparse families of coverage functions include:
\begin{itemize}
    \item 
The classical Beck-Fiala setting can be formulated as a $t$-sparse family of coverage functions, where for a hyperedge $H$, we define a coverage function
with $U = H$ and~$U_i = \{i\} \cap H$ for all $i\in [n]$.
\item The rank function of a partition matroid is a coverage function, where each item covers at most one element (assuming a capacity of $1$ for each set in the partition). 
Thus, a family of $m$ partition matroid rank functions is a $m$-sparse family.
\item For a given graph, the number of vertices covered by a set of edges is a coverage function, where each item (edge) covers at most two elements. Thus, a family of $m$ edge coverage functions forms a $2m$-sparse family.
\end{itemize}

\subsection{Our Contribution}

Towards the goal of understanding non-additive discrepancy for several natural classes of functions, we show the following results.

\paragraph*{Main Result}

Affirmatively answering all three previous questions, our main result is a constructive bound on the (multicolor) discrepancy of a $t$-sparse family of coverage functions, which is polynomial in~$t$, any color~$k$, and $\ln(n)$.

\begin{restatable}{theorem}{AllSets}
    \label{thm:all-sets-discrepancy}
    Let~$\F = \{f_1, \dotsc , f_m: \set{0,1}^n \rightarrow \R \}$ be a $t$-sparse family of coverage functions.
    In polynomial time, we can compute a $k$-coloring~$\chi : [n] \rightarrow [k]$ such that with high probability
    \begin{equation*}
        \disc_\chi(\F,k) < \OO \big(\tdiscrepancy\big).
    \end{equation*}
\end{restatable}

\paragraph*{Technical Overview}

Throughout the rest of the paper, we use an equivalent dual definition of coverage functions:
for each element $u\in U$, introduce a set $S_u\subseteq [n]$ that describes which items cover $u$.
Then we call $\Ss = \bigcup_{u \in U} \Ss_u$ the dual sets.
Formally, we redefine the $t$-sparse family of coverage functions as follows.

\begin{restatable}[$t$-Sparse Family of Coverage Functions]{definition}{tSparseCoverageFunction}
    \label{def:t-sparse-coverage-functions}
    Let~$m, t \in \N$.
    For all~$i\in[m]$, let~$\Ss_i$ be a collection of sets~$S \subseteq [n]$ such that each element~$j \in [n]$ appears in at most~$t$ different sets~$S \in \Ss := \bigcup_{i \in [m]} \Ss_i$.
    Let~$\F = \{f_1, \dotsc , f_m: \set{0,1}^n \rightarrow \R \}$ denote a $t$-sparse family of coverage functions defined by~$f_i(T) = \sum_{S \in \Ss_i} \min \set{ |S\cap T|, 1}$.
\end{restatable}

In order to show \cref{thm:all-sets-discrepancy}, we distinguish two special cases: 
either the sizes of the sets~$S\in \Ss$ are all small or all big.
Importantly, the results in this overview primarily provide intuition for the algorithmic challenges to overcome for the main result and are not fully optimized.

In the case of only small sets, we first construct a fractional $k$-coloring~$Y$ with few fractional elements and zero discrepancy based on the multilinear extension.
The key insight is that we can compute a large set for which the restricted multilinear extension becomes \emph{linear} (see \cref{def:restricted-multilinear-extension,lem:restricted-multilinear-extension-is-linear}), thereby allowing the use of standard LP techniques to round~$Y$ while maintaining zero discrepancy.
As a last step, we randomly round the remaining fractional elements in~$Y$ to an integral $k$-coloring without significantly increasing the discrepancy.
More precisely, we obtain the following discrepancy bound in \cref{sec:small-sets}.\footnote{For exposition, we omit optimization details in \cref{sec:small-sets} to focus on core ideas. Thus, the bound in \cref{thm:small-sets-discrepancy} depends on~$m$; this dependency can be removed by replacing the standard LP techniques by the more involved randomized procedure used for \cref{thm:all-sets-discrepancy}.}

\begin{restatable}{theorem}{SmallSets}
    \label{thm:small-sets-discrepancy}
    Let~$\F = \{f_1, \dotsc , f_m: \set{0,1}^n \rightarrow \R \}$ be a $t$-sparse family of coverage functions where~$\abs{S} \le s$ for all~$S \in \Ss$.
    In polynomial time, we can compute a $k$-coloring~$\chi : [n] \rightarrow [k]$ such that with high probability
    \begin{equation*} 
        \disc_\chi(\F,k) < \sqrt{2m t^3 k s \cdot (1 + \ln (2mk))}.
    \end{equation*}
\end{restatable}
    
In the case of only big sets, the key insight is that it is highly unlikely to miss a color under a random coloring. The sparsity of our setting ensures that failures in different sets are \emph{nearly independent} allowing for standard concentration bounds. For a slightly stronger result, the Constructive Lov\'asz  Local Lemma \cite{MoserT10, sason2026lovaszlocallemmafundamentals} guarantees the existence of a global coloring where no set is missing a color.  This yields the perfect discrepancy result shown in \cref{sec:big-sets}.

\begin{restatable}{theorem}{BigSets}
    \label{thm:big-sets-discrepancy}
    Let~$k \in \N$ and~$\F = \{f_1, \dotsc , f_m: \set{0,1}^n \rightarrow \R \}$ be a $t$-sparse family of coverage functions where~$\abs{S} \ge \sLargeSets$ for all~$S \in \Ss$.
    In randomized polynomial time, we can compute a $k$-coloring $\chi:[n]\rightarrow[k]$ such that~$\disc_\chi(\F,k) = 0$.
\end{restatable}

The techniques necessary to prove \cref{thm:small-sets-discrepancy,thm:big-sets-discrepancy} capture the main difficulties to overcome when showing the discrepancy bound for sets of all sizes in \cref{thm:all-sets-discrepancy}.
We start by introducing an artificial size threshold~$s$ to distinguish between small and big sets.
Remarkably, it suffices to ensure zero discrepancy \whp for a representative subset of size~$s$ for all big sets.
Combining the key ideas of the two cases, we start with an initial fractional $k$-coloring of zero discrepancy and iteratively increase the number of integral elements by a random procedure that maintains the value of the initial $k$-coloring in expectation while bounding the rounding error of the \emph{linear} restricted multilinear extensions \wrt the small sets.
Importantly, all elements of the same set are rounded \emph{independently}.
Putting the pieces together, we obtain a discrepancy bound \wrt the small sets from the rounding procedure and zero discrepancy \whp for the big sets by standard concentration bounds.
For more details, refer to \cref{sec:all-sets}.

\subsection{Future Directions} 

In \cref{thm:all-sets-discrepancy}, we show a discrepancy bound of~$\OO(\sqrt{t^3k} \cdot \log(nkt))$.
A simple lower bound shows that the discrepancy in our sparsity setting must be at least~$t$: Consider a single item and a single function, where the item covers~$t$ elements. 
In other words, it is not possible to color the item in more than one color, thus all other colors cannot be used which results in a discrepancy of~$t$.
Therefore, our bound is tight up to a factor of~$\OO(\sqrt{tk} \cdot \log(nkt))$.
In the following, we will comment on the two factors~$\OO(\log(nkt))$ and~$\OO(\sqrt{tk})$ separately.

The $\OO(\log(nkt))$ factor is inherent to our approach in order to ensure that the big sets have \whp zero discrepancy.
Avoiding this factor requires a new approach that does not rely on a union bound.
Similar logarithmic factors are lost in other
methods in the non-additive setting that are also based on random colorings, \eg both previous papers~\cite{HollenderMMS26,DupreTF25}.
Compared to the lower bound of $\Omega(\sqrt{m})$ by Meka and Manurangsi~\cite{ManurangsiM26} in the additive setting\footnote{There are no lower bounds specific to the non-additive setting.}, they lose a factor of~$\OO(\sqrt{\log (mk)})$ due to the use of concentration and union bounds.\footnote{To be precise, \cite{HollenderMMS26} actually loses a factor of~$\OO(\sqrt{k \log (mk)})$ due to their slightly weaker bound.}
It also remains an interesting open question whether the gap of~$\OO(\sqrt{\log (mk)})$ can be closed there.
Moreover, this extra factor is in some sense comparable to obtaining the standard bound of~$\OO(\sqrt{n \log n})$ in traditional discrepancy where randomly rounding a fractional $2$-coloring with only $n$ fractional elements is analyzed with concentration and union bounds as well.
Improving the bound to~$\OO(\sqrt{n})$ demanded the development of more refined rounding techniques~\cite{spencer1985six,Bansal10,Rothvoss17,LovettM15,BansalJM23,LevyRR17,PesentiV23} that do not require a union bound, however, heavily rely on additivity.
This raises the intriguing question whether the extra factor can be avoided in the non-additive setting as well.

The~$\OO(\sqrt{tk})$ factor is the overhead of our iterative approach. 
In fact, the naive analysis of the number of iterations results in about~$\OO(tk)$ which yields an overhead of~$\OO(tk)$ that we improve upon using Azuma's inequality.
We leave open whether this factor can be improved further.

Furthermore, observe that our notion of sparsity relies on the structure of coverage functions.
In an alternative definition of sparsity, as mentioned in~\cite{HollenderMMS26}, each item~$i$ influences only~$t$ functions, more precisely there exists some $S\subseteq [n]$ such that $f(S \cup \{i\}) \neq f(S)$. 
Under this notion, they show a non-constructive discrepancy bound of~$\OO((ntk)^{1/3} \cdot (\log(ntk))^{1/3})$ for any prime~$k$.
As this assumption is weaker than ours, it raises the intriguing question whether a similar bound holds for coverage or even broader function classes (under additional assumption of 1-Lipschitzness). In the case of partition matroid rank functions where both notions are equivalent, we affirmatively answer this question with an efficient algorithm to compute a $k$-coloring with no restrictions on~$k$.
Independent of the sparse setting, many questions
remain open in non-additive discrepancy. 
In this paper, we emphasize natural function classes
such as coverage functions, matroid rank functions, or generally monotone submodular functions that will form interesting benchmarks
for future research.

\section{Preliminaries}
\label{sec:preliminaries}


Throughout, we use the following notation.
We write~$[n] = \set{1, \dotsc, n}$ for any $n \in \N$.
For any set~$S \subseteq [n]$, denote its complement by~$\conj{S} := [n] \backslash S$ and its indicator vector by~$\One_S \in \set{0,1}^n$.

An integral \emph{$k$-coloring} of $n$ elements is a function $\chi : [n] \rightarrow [k]$ that assigns each element exactly one color.
A \emph{fractional $k$-coloring} is a function $X : [n] \rightarrow \Delta_k$ with the $k$-simplex $\Delta_k = \set{x \in [0,1]^k  \,|\, \sum_{\ell \in [k]} x_\ell = 1}$.
For any $k$-coloring~$\chi$, the elements of color $\ell \in [k]$ are interchangeably written as~$\chi^{-1}(\ell)$ and the vector~$\chi_\ell \in \set{0,1}^n$ where~$\chi_{\ell j} = 1$ for any~$j \in [n]$ with~$\chi_j = \ell$.
Similarly, for any fractional $k$-coloring~$X$, we use the vector~$X_\ell \in [0,1]^n$ for all elements of color $\ell \in [k]$.
Slightly abusing notation, we also use the vector~$X_j \in [0,1]^k$ with $ \sum_{\ell \in [k]} X_{j \ell} = 1$ for the coloring of a single element~$j \in [n]$.


\subsection{Rounding the Linear Multilinear Extension}

We extend our discrete functions to the continuous setting via the \emph{multilinear extension}.

\begin{definition}[Multilinear Extension]
    \label{def:multilinear-extension}
    For any function~$f:\set{0,1}^n \rightarrow \R$, its multilinear extension~$F:[0,1]^n\rightarrow\R$ is defined as
    \begin{equation*}
        F(x)= \E_{S\sim x}[f(S)] = \sum\nolimits_{S\subseteq [n]}f(S)\cdot \prod\nolimits_{i\in S}x_i\prod\nolimits_{j\in \conj{S}}(1-x_j).
    \end{equation*}
\end{definition}

Importantly, in the case of coverage functions, one can evaluate the multilinear extension in time polynomial in the size of the underlying universe, since it is straight-forward to compute the probability that an element of the universe is covered.
In \cref{sec:small-sets,sec:all-sets}, we consider the following restricted multilinear extension.

\begin{definition}[Restricted Multilinear Extension]
    \label{def:restricted-multilinear-extension}
    Let~$f :\set{0,1}^n \rightarrow \R$ be a function
    and~$X : [n] \rightarrow \Delta_k$ a fractional $k$-coloring.
    For any set~$D \subseteq [n]$, denote the restricted multilinear extension by~$F \vert_D :[0,1]^{\abs{D}} \rightarrow \R$ where the values~$X_{\bar{d}}$ are fixed for all~$\bar{d} \in \conj{D}$.
\end{definition}

In words, restricting the multilinear extension to elements in~$D$ reduces the number of variables to~$\abs{D}$ and treats the values of the elements in~$\conj{D}$ as constants.
Importantly, the restricted multilinear extension becomes \emph{linear} for the following choice of~$D$.

\begin{restatable}{lemma}{RestrictedMultilinearExtension}
    \label{lem:restricted-multilinear-extension-is-linear}
    Let~$f:\set{0,1}^n \rightarrow \R$ be a
    coverage function with dual sets $\Ss$ and let~$X : [n] \rightarrow \Delta_k$ a fractional $k$-coloring.
    For any set~$D \subseteq [n]$ with~$\abs{S \cap D} \le 1$ for all~$S \in \Ss$, the restricted multilinear extension~$F \vert_D (X)$ is linear.
\end{restatable}

We defer the proof to Appendix~\ref{sec:appendix}.
In the following, we show how to take advantage of this insight when showing small discrepancy bounds.
Given a fractional $k$-coloring~$X : [n] \rightarrow \Delta_k$, let~$F_i \vert_{D}$ denote the restricted multilinear extension \wrt~$X$ for all~$i \in [m]$.
By \cref{lem:restricted-multilinear-extension-is-linear}, all functions~$F_i\vert_{D}$ are linear. 
Thus, the problem of coloring all elements in~$D$ while preserving the function values~$F_i\vert_{D}(X_{\ell})$ can be modeled as the following linear program \emph{RoundingLP}.
    \begin{alignat*}{3}
        & \forall i \in [m], \ell \in [k] \qquad \qquad \qquad & F_i\vert_{D}(X'_{\ell})  & = F_i\vert_{D}(X_{\ell}) \qquad \qquad \qquad && \text{\descriptions Value Preservation} \\
        & \forall d \in D \qquad \qquad \qquad & \sum\nolimits_{\ell \in [k]} X'_{\ell d}  & = 1  \qquad \qquad \qquad && \text{\descriptions Valid Coloring} \\
        & \forall d \in D, \ell \in [k] \qquad \qquad \qquad & X'_{\ell d} & \ge 0 \qquad \qquad \qquad && \text{\descriptions Non-Negativity}
    \end{alignat*}
Excluding the non-negativity constraints, RoundingLP has~$mk + \abs{D}$ many constraints.
As the values for all~$\bar{d} \in \conj{ D}$ are fixed, RoundingLP has~$\abs{D} \cdot k$ many variables~$X'_{\ell d}$ for all~$d \in  D$ and~$\ell \in [k]$.
In \cref{sec:all-sets}, we use the following lemma and RoundingLP to round a fractional $k$-coloring~$X$ to a more integral $k$-coloring~$X'$ while maintaining the value of~$X$ in expectation.

\begin{restatable}[Rounding in Expectation]{lemma}{RoundingExpectation}
    \label{lem:rounding-with-expectation}
    Let $A \in \R^{m \times n}$ and $b \in \R^m$.
    Given $X \in [0,1]^n$ with~$AX = b$, we can sample a random vector~$X' \in [0,1]^n$ with $\E[X']= X$, $AX' = b$ and at most~$m$ fractional variables in polynomial time.
\end{restatable}

We defer the proof to Appendix~\ref{sec:appendix}.


\subsection{Concentration Bounds}

In order to bound the randomized rounding error in \cref{sec:small-sets}, we use McDiarmid's inequality.

\begin{lemma}[McDiarmid's Inequality \cite{mcdiarmid1989method}]
    \label{lem:mcDiarmid-inequality}
    Let~$Y = (Y_1, \dotsc, Y_q) \in \set{0,1}^q$ be a vector of independent random variables.
    Let~$h : \set{0,1}^q \rightarrow \R$ be any function satisfying $\abs{h(Y)- h(Y')} \le t$ whenever~$Y$ and~$Y'$ differ in at most one coordinate.
    For any $r > 0$, we have
    \begin{equation*}
        \Pr \big(\abs{h(Y) - \E[h(Y')]} \ge r \big) \le 2 \exp \left(- \frac{2r^2}{qt^2}\right).
    \end{equation*}
\end{lemma}

We use the following concentration bounds to show the discrepancy bound in \cref{sec:all-sets}.

\begin{lemma}[Chernoff Bound] 
    \label{lem:bernsteins-inequality}
    Let~$X_1, \dotsc ,X_n \in \set{0,1}$ be independent random variables.
    Let~$X = \sum_{i \in [n]} X_i$ and~$\mu = \E[X]$.
    For any~$\delta \in [0,1]$, we have
    \begin{equation*}
        \Pr \big( (1-\delta) \mu \geq X \big) \leq \exp\left(-1/2 \cdot \delta^2\mu\right).
    \end{equation*}
\end{lemma}

We call a sequence $(M_t)_{t\in\N}$ of random variables a \emph{discrete-time martingale} if we have $\E[M_t]< \infty$ and $\E[M_{t+1}\,\lvert\, M_0,...,M_t]=X_t$ for all~$t\in\N$.
\begin{lemma}[Azuma's Inequality]\label{lem:azumas-inequality}
    Suppose $(M_j)_{j\geq0}$ is a martingale and $\abs{M_j-M_{j-1}}\leq c_j$ almost surely for all $j\geq1$. Then, for all positive integers $N$ and all positive reals $\eps$,
    \begin{equation*}
        \Pr(\abs{X_N-X_0}\geq\eps)\leq 2\exp\left(-\frac{\eps^2}{2\sum_{j=1}^Nc_j^2}\right).
    \end{equation*}
\end{lemma}

\section{Sparse Coverage Functions with Small Sets}
\label{sec:small-sets}

In this section, we show the following theorem. 

\SmallSets*

We show \cref{thm:small-sets-discrepancy} in two steps: First, we construct a fractional $k$-coloring~$Y $ with at most~$\fracEntries$ fractional elements and zero discrepancy.
Second, we randomly round the remaining fractional elements to an integral $k$-coloring~$\chi$ without significantly increasing the discrepancy.


\subsection{Computing a k-Coloring with Few Fractional Elements}
\label{subsec:small-sets-sparse-frac-sol}

Throughout this subsection, we show how to compute a fractional $k$-coloring with few fractional elements by proving the following lemma.

\begin{lemma}
    \label{lem:small-sets-sparse-frac-sol}
    Let~$\F = \{f_1, \dotsc , f_m: \set{0,1}^n \rightarrow \R \}$ be a $t$-sparse family of coverage functions where~$\abs{S} \le s$ for all~$S \in \Ss$.
    In polynomial time, we can compute a fractional $k$-coloring $Y : [n] \rightarrow \Delta_{k}$ with at most $\fracEntries$ fractional elements such that $F_i(Y_{\ell}) = F_i(Y_{\ell'})$ for all~$i \in [m]$ and~$\ell, \ell' \in [k]$.
\end{lemma}

We prove \cref{lem:small-sets-sparse-frac-sol} as follows.
Starting out with a fractional $k$-coloring $Y : [n] \rightarrow \Delta_{k}$ where~$Y_{\ell j} = 1/k$ for all $j \in [n]$ and $\ell \in [k]$, we iteratively increase the number of integrally colored elements until we have at most~$\fracEntries$ fractional elements.
Throughout, let~$Z \subseteq [n]$ denote the set of fractional elements in~$Y$, where~$Z = [n]$ initially.
In the following lemma, we show how to compute a large independent set in the dependency graph of~$\F$.

\begin{lemma}
    \label{lem:small-sets-frac-independent-set}
    For some~$Z \subseteq [n]$, let $G = (Z,E)$ denote the dependency graph of~$\F$ where~$E = \set{\set{u,v} \,|\, \exists S \in \Ss:  u,v \in S \cap Z, u \neq v}$. 
    In polynomial time, we can compute an independent set~$D \subseteq Z$ in~$G$ such that~$|D| \ge \abs{Z}/ts$.
\end{lemma}

\begin{proof}
    In polynomial time, we construct~$G$ and compute~$D$ using a standard greedy algorithm.
    Initialize~$D=\emptyset$.
    While~$Z\neq\emptyset$, we repeat the following steps.
    \begin{enumerate}
        \item Pick any vertex $z\in Z$ with the smallest degree.
        \item Set $D = D \cup \set{z}$ and remove $z$ and all its neighbors from $Z$.
    \end{enumerate}
    Due to the second step, it holds that~$\set{d_1,d_2} \not \in E$ for all $d_1,d_2 \in D$.
    By definition of size~$s$ and  sparsity~$t$, it holds that~$\deg(z) \le (t - 1)(s-1)$ for all vertices~$z \in Z$.
    Hence, we delete~$z$ and at most~$(t - 1)(s-1)$ dependent vertices from~$Z$ in each iteration.
    Thus, it follows that~$\abs{D}\geq \abs{Z}/ts$.
\end{proof}

In words, no two elements in~$D$ appear in the same set~$S \in \Ss$.
In the following, we explain how to exploit the independence of the elements of~$D$ to compute a fractional $k$-coloring~$Y'$ with at most~$\abs{Z}-1$ many fractional elements while maintaining~$F_i (Y_{\ell}) = F_i (Y'_{\ell})$ for all~$i \in [m]$ and~$\ell\in [k]$ with \cref{lem:small-sets-round-multilinear-extension}.
The key insight is that restricting the multilinear extension~$F_i \vert_{D}$ to~$D$ and fixing all elements~$\bar{d} \in \conj{ D}$ to their value of~$Y_{\bar{d}}$ results in~$F_i \vert_{D}$ becoming \emph{linear} (see \cref{def:restricted-multilinear-extension,lem:restricted-multilinear-extension-is-linear}).
Therefore, the problem of coloring the elements of~$D$ while ensuring that~$F_i\vert_{D}(Y'_\ell) = F_i\vert_{D}(Y_{\ell})$ for all~$i \in [m]$ and~$\ell \in [k]$ can be modeled as a linear program.
In the following lemma, we show that standard LP techniques result in at least one more integrally colored element in~$D$.

\begin{lemma}
    \label{lem:small-sets-round-multilinear-extension}
    Let~$Y : [n] \rightarrow \Delta_{k}$ be a fractional $k$-coloring and~$Z \subseteq [n]$ the set of fractional elements in~$Y$.
    Let~$D \subseteq Z$ be an independent set of size~$\abs{D} > mk$.
    In polynomial time, we can compute a fractional $k$-coloring $Y' : [n] \rightarrow \Delta_{k}$ with at most $\abs{Z} - 1$ fractional elements and~$F_i(Y_{\ell}) = F_i(Y'_{\ell})$ for all~$i \in [m]$ and~$\ell \in [k]$. 
\end{lemma}

\begin{proof}
    In the following, we show how to obtain~$Y'$ such that
    \begin{alignat}{3}
        & \forall \bar{d} \in \conj{D} \qquad \qquad \qquad & Y'_{\bar{d}} &= Y_{\bar{d}}, \label{eq:small-sets-fixed-fractional-elements}\\
        & \forall i \in [m], \ell \in [k] \qquad \qquad \qquad & F^s_i \big(Y'_\ell \big) &= F^s_i \big(Y_{\ell} \big) \label{eq:small-sets-value-preservation} \\
        & \exists d \in  D \qquad \qquad \qquad & Y'_{d} &\in \set{0,1}^{k} \text{ with } \sum\nolimits_{\ell \in [k]} Y'_{d \ell} = 1. \label{eq:small-sets-at-least-one-integral-element}
    \end{alignat}
    To satisfy \cref{eq:small-sets-fixed-fractional-elements}, we set~$Y'_{\bar{d}} = Y_{\bar{d}}$ for all~$\bar{d} \in \conj{D}$.
    For all~$i \in [m]$, let~$F_i \vert_{D}$ denote the restricted multilinear extension \wrt~$Y$.
    Then we use RoundingLP to compute a fractional coloring~$(Y'_{d})_{d \in D}$ such that \cref{eq:small-sets-value-preservation,eq:small-sets-at-least-one-integral-element} hold.
    Due to~$\abs{D} > mk$, we have more variables than constraints, \ie we have~$\abs{D} + mk < \abs{D} \cdot k$.
    Using a standard LP algorithm, we obtain a vertex solution~$(Y'_{d})_{d \in D}$ to RoundingLP with at most~$ \abs{D} + mk$ many non-zero variables in polynomial time. 
    By the valid coloring constraints, any element~$d \in  D$ must have at least one non-zero variable.
    If~$d$ is fractionally colored, then at least two variables must be non-zero.
    Hence, there can be at most~$mk$ many elements that have more than one non-zero variable.
    As a consequence, at least $\abs{D} - mk \ge 1$ elements of~$D$ are integrally colored and \cref{eq:small-sets-at-least-one-integral-element} is satisfied.
    
    Using \cref{def:restricted-multilinear-extension}, it holds that~$F_i \vert_{ D}(Y_{\ell}) = F_i (Y_{\ell})$
    for a fixed fractional $k$-coloring~$Y_{\ell}$ and all~$i \in [m]$ and~$\ell \in [k]$.
    Together with the value preservation constraints of RoundingLP, we have that \cref{eq:small-sets-value-preservation} holds.
\end{proof}

We repeatedly reduce the number of fractional elements using \cref{lem:small-sets-frac-independent-set,lem:small-sets-round-multilinear-extension} until we can no longer find an independent set~$ D$ of size~$\abs{D} > mk$.
At that point, the resulting $k$-coloring has at most~$\fracEntries$ fractional elements.
We are now in the position to prove \cref{lem:small-sets-sparse-frac-sol}.

\begin{proof}[Proof of \cref{lem:small-sets-sparse-frac-sol}]
    Let~$Y : [n] \rightarrow \Delta_{k}$ denote a fractional $k$-coloring with~$Y_{j\ell} = 1/k$ for all $j \in [n]$ and $\ell \in [k]$.
    Due to the choice of~$Y$, we have zero discrepancy
    \begin{equation}
        \label{eq:small-sets-zero-discrepancy}
        F_i (Y_{\ell}) = F_i(Y_{\ell'})
    \end{equation}
    for all~$i \in [m]$ and~$\ell, \ell' \in [k]$.
    Let~$Z \subseteq [n]$ be the set of fractional elements in~$Y$, initially we have~$Z = [n]$.
    In order to compute a fractional $k$-coloring $Y^*$ with at most $\fracEntries$ fractional elements such that \cref{eq:small-sets-zero-discrepancy} holds, we repeat the two following steps.
    First, we find an independent set~$D \subseteq Z$ of size~$\abs{ D} \ge \abs{Z}/ts$ with \cref{lem:small-sets-frac-independent-set}.
    Second, we compute a fractional $k$-coloring~$Y'$ with at most~$\abs{Z} -1$ many fractional elements while maintaining
    \begin{equation}
        \label{eq:small-sets-zero-discrepancy-in-each-iteration}
        F_i (Y'_{\ell}) = F_i (Y_{\ell})
    \end{equation}
    for all~$i \in [m]$ and~$\ell\in [k]$ using \cref{lem:small-sets-round-multilinear-extension}.
    Then we pick~$Y^*$ to be the $k$-coloring that is the result of repeating these steps until there no longer exists a set~$D$ with~$\abs{D} > mk$.
    At this point, $Y^*$ has at most~$\fracEntries$ fractional elements.
    From \cref{eq:small-sets-zero-discrepancy,eq:small-sets-zero-discrepancy-in-each-iteration}, it holds that
    \begin{equation*}
        F_i (Y'_{\ell}) 
        = F_i (Y_{\ell})
        = F_i(Y_{\ell'}) 
        = F_i (Y'_{\ell'})
    \end{equation*}
    for all~$i \in [m]$ and~$\ell, \ell' \in [k]$.
    Consequently, maintaining \cref{eq:small-sets-zero-discrepancy-in-each-iteration} throughout the algorithm results in~$F_i (Y^*_{\ell}) = F_i(Y^*_{\ell'})$ for all~$i \in [m]$ and~$\ell, \ell' \in [k]$.
\end{proof}

    
\subsection{Rounding a k-Coloring with Few Fractional Elements}
\label{subsec:small-sets-randomized-rounding}

Based on the fractional $k$-coloring computed in \cref{subsec:small-sets-sparse-frac-sol}, we show that a simple randomized rounding scheme, similar to \cite{DupreTF25,HollenderMMS26}, yields a proper $k$-coloring of small discrepancy.

\begin{restatable}{lemma}{DiscrepancyAnalysis}
    \label{lem:small-sets-randomized-rounding-discrepancy}
    Let $\F = \{f_1, \dotsc , f_m: \set{0,1}^n \rightarrow \R \}$ be a $t$-sparse family of coverage functions.
    Let~$Y : [n] \rightarrow \Delta_{k}$ be a fractional $k$-coloring such that~$F_i(Y_{\ell}) = F_i(Y_{\ell'})$ for all $i \in [m]$ and~$\ell, \ell' \in [k]$ and~$Z \subseteq [n]$ the set of fractional elements in~$Y$.
     In polynomial time, we can compute a $k$-coloring~$\chi : [n] \rightarrow [k]$ such that with probability larger than~$1/2$ holds that
    \begin{equation*}
        \disc_\chi(\F,k) < \sqrt{2 \abs{Z}t^2 \cdot (1 + \ln (2mk))}.
    \end{equation*}
\end{restatable}

We defer the proof to \cref{sec:appendix}.
Together, \cref{lem:small-sets-sparse-frac-sol,lem:small-sets-randomized-rounding-discrepancy} show \cref{thm:small-sets-discrepancy}.

\section{Sparse Coverage Functions with Big Sets}
\label{sec:big-sets}
In this section, we show the following theorem. 

\BigSets*

There are two key insights to proving \cref{thm:big-sets-discrepancy}.
First, any $k$-coloring~$\chi$ where all sets~$S \in \Ss$ contain at least one element of each color, \ie $1_S \cdot \chi_\ell \ge 1$ for all $\ell \in [k]$, has zero discrepancy.
Any set~$S$ with that property is called a \emph{$k$-rainbow set}.

\begin{restatable}[$k$-Rainbow Set]{definition}{RainbowSet}
    \label{def:rainbow-sets}
    Let~$\chi : [n] \rightarrow [k]$ be a $k$-coloring. A set~$S\subseteq[n]$ is called a \emph{$k$-rainbow set} with respect to $\chi$, if $\One_S \cdot \chi_\ell \ge 1$ for all $\ell \in [k]$.
    If a subset~$P \subseteq S$ is a $k$-rainbow set with respect to~$\chi$, then so is~$S$. 
\end{restatable}

Second, consider a $k$-coloring~$\chi$ where each element is randomly colored in one of the~$k$ colors.
Due to the sparsity~$t$, the random coloring of all elements in~$S \in \Ss$ is only dependent on a few other sets.
Since the sets are large in comparison to the number of dependencies, the probability of a set being a rainbow set in a random coloring is high and we can apply the \emph{Lov\'{a}sz Local Lemma (LLL)}.

\begin{lemma}[Symmetric Constructive Lov\'asz Local Lemma \cite{MoserT10, sason2026lovaszlocallemmafundamentals}]
    \label{lem:big-sets-LLL}
    Let $\mathcal X$ be a finite set of mutually independent random binary variables in a probability space.
    Let $\A$ be a finite set of events determined by these variables. 
    Let $p\in [0,1]$ and $d\in\mathbb N$ such that
    $\Pr(A) \le p$ for all~$A\in \A$ and 
    every $A\in \A$ depends on at most $d$ other events $A'\in \A$. 

    If $ep(d+1) \le 1$, then there is a non-zero probability that none of the bad events occur.
    Moreover, there exists a randomized algorithm that computes such an realization of random variables in expected time polynomial in $|\mathcal X|$ and $|\A|$.
\end{lemma}

Putting the two insights together, we show the following lemma as follows. 
For each set $S\in \Ss$, we define the bad event as $S$ not being a rainbow set in a random coloring and analyze its probability and dependencies. 
Applying LLL, we show that all bad events can be avoided and we can obtain a coloring where each set is a rainbow set. 
Importantly, we first replace each set~$S$ by arbitrary subset~$P$ of size~$s = \sLargeSets$. 
This step maintains sufficiently low probability for the bad events, but restricts the dependencies as a larger set may have higher dependencies.

\begin{lemma}
    \label{lem:big-sets-LLL-zero-discrepancy-k-colors}
    For any~$t,k \in \N$, let~$s= \sLargeSets$.
    Let~$\Ss$ be a collection of sets~$S \subseteq [n]$ such that $\abs{S}\geq s$ and each element~$j \in [n]$ appears in at most~$t$ different sets~$S \in \Ss$.
    In randomized polynomial time, we can compute a $k$-coloring $\chi:[n]\rightarrow[k]$ such that each~$S \in \Ss$ is a $k$-rainbow set.
\end{lemma}
\begin{proof}
    Let~$\Pp$ be a collection of sets that contains for each~$S \in \Ss$ an arbitrary~$P \subseteq S$ of size~$\abs{P} = s$. 
    Then $\Pp$ is also $t$-sparse.
    Let~$\chi:{[n]}\rightarrow[k]$ be a $k$-coloring where each element~$j \in [n]$ is randomly colored by setting $\chi_{j} = \ell$ with probability~$1/k$ for each~$\ell \in [k]$.
    For all $P \in \Pp$ and~$\ell \in [k]$, let $A(P)$ denote the event that $P$ is not a $k$-rainbow set with respect to~$\chi$ and let~$C_\ell(P)$ denote the event that~$1_P \cdot \chi_\ell < 1$.
    Then we have
    \begin{equation*}
        \Pr \big( C_\ell(P) \big)=\left(1-\tfrac{1}{k}\right)^s \le \exp\left(-\tfrac{s}{k}\right).
    \end{equation*}
    Notably, if~$P \in  \Pp$ is not a $k$-rainbow set, then~$1_P \cdot \chi_\ell < 1$ for at least one color $\ell \in[k]$. 
    Therefore, using a standard union bound, it holds for all $P \in \Pp$ that
    \begin{equation*}
        p := \Pr \big( A(P) \big)
        = \Pr \left(\bigcup\nolimits_{\ell \in[k]}  C_\ell(P) \right)
        \leq \sum\nolimits_{\ell \in[k]} \Pr \big( C_\ell(P) \big) 
        = k\cdot\left(1-\tfrac{1}{k}\right)^s 
        \le k\cdot\exp\left(-\tfrac{s}{k}\right).
    \end{equation*}
    Here, the last inequality follows from $1+x \le \exp(x)$, which holds for every $x\in\R$.

    Since each element $j \in [n]$ appears in at most $t$ different sets~$P \in \Pp$, each event $A(P)$ is only dependent on at most $d = s(t-1) \le ts - 1$ other events. 
    Thus, we have that
    \begin{align*}
        e p (d + 1) 
        &= etk \cdot \exp\left(-\frac{s}{k}\right) \cdot s 
        = 6etk^2\cdot \exp(-\ln(t^6k^6)) \cdot \ln(tk)
        = \frac{6e \cdot \ln(tk)}{t^5k^4}
        \leq 1,
    \end{align*}
    where the last inequality holds for all $k\geq2$ and $t\geq1$.
    Therefore, we can compute in polynomial time a $k$-coloring~$\chi'$ such that no bad event occurs. In particular, all sets~$P \in \Pp$ and therefore also all sets $S\in \Ss$ are $k$-rainbow sets with respect to~$\chi'$.
\end{proof}

\cref{thm:big-sets-discrepancy} is the immediate consequence of \cref{lem:big-sets-LLL-zero-discrepancy-k-colors}.

\begin{proof}[Proof of \cref{thm:big-sets-discrepancy}]
    By \cref{lem:big-sets-LLL-zero-discrepancy-k-colors}, we can compute a $k$-coloring $\chi:[n]\rightarrow[k]$ where each set~$S \in \Ss$ is a $k$-rainbow set in polynomial time.
    Hence, it holds that
    \begin{equation*}
        \min\{\One_{S} \cdot \chi_\ell , 1\} 
        \ge \min\{1 , 1\} 
        = 1
    \end{equation*}
    for all~$S \in \Ss$ and $\ell \in [k]$.
    Therefore, we have $f_i(\chi_\ell) = f_i(\chi_{\ell'})$ for all~$\ell \in [k]$ and~$i \in [m]$.
    As a result, it holds that~$\disc_\chi(\F,k) = 0$.
\end{proof}

\section{Sparse Coverage Functions with Any Size Sets}
\label{sec:all-sets}

Throughout this section, we show the following theorem. 

\AllSets*

We start by explaining the key ideas of computing the integral $k$-coloring~$\chi$ of \cref{thm:all-sets-discrepancy}.
Based on an artificial size threshold $s$, we partition our family of functions~$\F$ into two parts.

\begin{definition}
    \label{def:all-sets-small-big-set-functions}
    Let~$s = \sAllSetst$ and~$\F = \{f_1, \dotsc , f_m: \set{0,1}^n \rightarrow \R \}$.
    For all~$i \in [m]$, let~$\Ss_i = \Ss^s_i \, \dot \cup \, \Ss_i^b$ where~$\Ss^s_i$ is the collection of sets~$S \in \Ss_i$ with~$\abs{S} \le s$ and~$\Ss^b_i$ the collection of sets~$S' \in \Ss_i$ with~$\abs{S'} > s$.
    Based on~$\Ss^s_i$ and~$\Ss^b_i$, write~$f_i = f^s_i + f^b_i$ and define the $t$-sparse families~$\F^s = \set{f^s_1, \dotsc , f^s_m}$ and~$\F^b = \set{f^b_1, \dotsc , f^b_m}$.
\end{definition}

Intuitively, $\F^s$ is the $t$-sparse family of \emph{small set} coverage functions and $\F^b$ is the $t$-sparse family of \emph{big set} coverage functions.
Let~$Y^{(0)}$ be the fractional $k$-coloring where~$Y^{(0)}_{\ell q} = 1/k$ for all~$q \in [n]$ and~$\ell \in [k]$ as an initial solution, then we have zero discrepancy
\begin{equation*}
    F^s_i \left( Y^{(0)}_{\ell} \right) = F^s_i \left( Y^{(0)}_{\ell'} \right)
\end{equation*}
for all~$i \in [m]$ and~$\ell, \ell' \in [k]$.
Next, we iteratively increase the number of integral elements in our fractional $k$-coloring by a random procedure that maintains the value of~$Y^{(0)}$ in expectation while bounding the rounding error of~$F^s_i$ in each iteration.
As a result, the integral $k$-coloring~$\chi$ has small discrepancy \wrt~$\F^s$.

Notably, this approach only considers the discrepancy of the \emph{small sets} in~$\Ss^s$.
This raises the question of whether our approach yields any discrepancy property for the \emph{big sets} in~$\Ss^b$.
Observe that if each set in~$\Ss^b$ is with high probability a $k$-rainbow set \wrt~$\chi$ (see \cref{def:rainbow-sets}), then this implies zero discrepancy \wrt~$\F^b$ with high probability.
Combining these two ideas, we show the following lemma throughout this section.

\begin{lemma}
    \label{lem:all-sets-small-set-discrepancy}
    For any~$t,k \in \N$, let~$n > \Omega (t \cdot \ln(tk))$.
    Let~$s = \sAllSetst$ and $\F^s = \{f^s_1, \dotsc , f^s_m: \set{0,1}^n \rightarrow \R \}$ be a $t$-sparse family of small set coverage functions.
    In polynomial time, we can compute a $k$-coloring $\chi : [n] \rightarrow [k]$ such that with high probability each~$S \in \Ss^b$ is a $k$-rainbow set \wrt~$\chi$ and 
    \begin{equation}
        \label{eq:all-sets-discrepancy-bound}
        \disc_\chi(\F^s,k) < \OO\big(\tdiscrepancy\big).
    \end{equation}
\end{lemma}

Due to the size of~$s$, it suffices to show that integrally coloring $s$ elements of each~$S \in \Ss^b$ results in~$S$ being a $k$-rainbow set with high probability \wrt~$\chi$.
For the rest of the section, assume \wlogeneral that all big sets~$S \in \Ss^b$ have exactly~$s$ elements by arbitrarily picking~$s$ elements from~$S$ and removing the rest.
As a consequence, we slightly abuse notation and define~$\F^s$ over~$\Ss^s \cup \Ss^b$, because all sets in~$\Ss^s \cup \Ss^b$ have at most~$s$ elements.

We show \cref{lem:all-sets-small-set-discrepancy} in two steps.
First, we compute~$I \le ts + 1$ many sets~$D_1, \dotsc, D_{I} \subseteq [n]$ where no two elements of~$D_j$ appear in the same set~$S \in \Ss^s \cup \Ss^b$.
This is equivalent to finding a~$ts + 1$ vertex coloring in the dependency graph~$G$ of~$\F^s$ where two distinct elements share an edge if they are present in the same set~$S$.

Second, we iteratively increase the integral elements in the initial solution~$Y^{(0)}$ by rounding the fractional elements of one independent set~$D_j$ in each iteration~$j \in [I]$  while all other elements in~$\conj{D_j}$ remain unchanged.
More precisely, in the $j$-th iteration the elements~$Y^{(j)}_d$ for all~$d \in D_j$ become integral with~$\E[ Y^{(j)}_{d}] = Y^{(j-1)}_{d}$ while maintaining~$Y^{(j)}_{\bar{d}} = Y^{(j-1)}_{\bar{d}}$ for all elements~$\bar{d} \in \conj{D_j}$.
Moreover, we also bound the error in the linear multilinear extensions~$F^s_i$ in each iteration.
The key insight to a small error is the following:
Similar to Beck-Fiala \cite{BeckF81}, we only consider linear functions~$F^s_i\vert_{D_j}$ with more than~$\tThreshold$ occurences of fractional elements with respect to~$D_j$ and disregard all others.
Thus, we only make an error once a function has at most~$\tThreshold$ fractional variables.
Formally, we show the following lemma.

\begin{lemma}
    \label{lem:all-sets-rounding-multilinear-extension}
    Let~$Y : [n] \rightarrow \Delta_{k}$ be a fractional $k$-coloring and~$D \subseteq [n]$ an independent set.
    In polynomial time, we compute a fractional $k$-coloring $Y' : [n] \rightarrow \Delta_{k}$ such that
    \begin{alignat*}{3}
        & \forall \bar{d} \in \conj{D} \qquad \qquad \qquad  & Y'_{\bar{d}} &= Y_{\bar{d}}, \\
        & \forall d \in D \qquad \qquad \qquad & \E\big[Y'_{d} \big] & = Y_{d}, \\
        & \forall d \in  D \qquad \qquad \qquad \qquad \qquad & Y'_{d} \in \set{0,1}^{k} \text{ with } \sum\nolimits_{\ell \in [k]} Y'_{d \ell} &= 1, \\
        & \forall i \in [m], \ell \in [k] \qquad \qquad \qquad & \left| F^s_i \big(Y'_{\ell} \big) - F^s_i \big(Y_{\ell} \big) \right| &\le \tThreshold.
    \end{alignat*}
\end{lemma}

\begin{proof}
    We show how to compute~$Y'$ by describing an iterative rounding algorithm that repeatedly applies \cref{lem:rounding-with-expectation} to RoundingLP in order to integrally color all elements~$d \in D$ while maintaining~$\E[Y'_{d}] = Y_{d}$.
    For all~$\bar{d} \in \conj{ D}$, we set~$Y'_{\bar{d}} = Y_{\bar{d}}$.
    Let~$r$ denote the number of linear functions~$F^s_i \vert_{D}$ defined \wrt~$Y$ with more than~$3t$ fractional elements \wrt~$D$ and assume \wlogeneral that these functions are~$F_1, \dotsc, F_r$.
    
    Let $J \le \abs{D}$ denote the number of iterations of the following algorithm.
    Starting out with~$X^{(0)} = (Y_{d})_{d \in D} \in (0,1)^{\abs{D} \cdot k}$ as an initial solution, we apply \cref{lem:rounding-with-expectation} to a subset of constraints of RoundingLP in each iteration.
    Let~$D_0 = D$.
    For all~$j \in [J]$, let~$D_j \subset D_{j-1}$ be the set of fractional elements in the solution~$X^{(j)} \in [0,1]^{\abs{D} \cdot k}$. 
    Let~$r_j$ be the number of functions $F^s_i \vert_{D}$ with more than~$\tThreshold$ occurrences of fractional elements \wrt~$D_j$, \ie all occurrences of a fractional element in~$F^s_i \vert_{D}$ are counted towards the threshold.
    Observe that~$r_j \le \abs{D_j}t/(\tThreshold + 1)$.
    In each iteration~$j \in [I]$, we omit the value preservation constraints for all functions~$F^s_i \vert_{D}$ with at most~$\tThreshold$ fractional elements \wrt~$D_j$ and the valid coloring constraints for the integral elements in $D_{j-1} \setminus D_j$.
    Therefore, we have~$r_jk + \abs{D_j}$ remaining constraints and~$\abs{D_j} \cdot k$ variables.
    Furthermore, it holds for all~$j \in [I]$ that
    \begin{align*}
        r_j k + \abs{D_j} 
        &\le \left\lfloor \frac{\abs{D_j}t}{\tThreshold + 1} \right\rfloor \cdot k + \abs{D_j} 
        < \frac{\abs{D_j}tk}{\tThreshold} + \abs{D_j}
        = \left(\frac{k}{3} + 1 \right) \cdot \abs{D_j} 
        < \abs{D_j} \cdot k.
    \end{align*}
    Using \cref{lem:rounding-with-expectation}, solution~$X^{(j+1)}$ has at least~$\abs{D_j} \cdot (k-1) - r_j k \ge 1$ more integral elements than~$X^{(j)}$ and we have~$\E[X^{(j+1)}] = X^{(j)}$.

    After~$J$ iterations, we have~$D_{J+1} = \emptyset$ and~$X^{(J)} \in \set{0,1}^{\abs{D}\cdot k}$.
    For all~$d \in D$, set $Y'_d = X^{(J)}_d$.
    For a fixed~$i \in [r]$, let~$\iota \in [I]$ denote the last iteration in which function~$F^s_i\vert_{D}$ has more than~$\tThreshold$ fractional elements \wrt~$D_{\iota}$.
    Using the triangle inequality, we have that
    \begin{align*}
        \big| F^s_i\vert_{D}\big(Y'_{\ell} \big) - F^s_i\vert_{D}\big(Y_{\ell}\big) \big| 
        &\le \big| F^s_i\vert_{D}\big(Y'_{\ell} \big) - F^s_i\vert_{D}\big(X^{(\iota)}_{\ell}\big) \big| + \big| F^s_i\vert_{D}\big(X^{(\iota)}_\ell \big) - F^s_i\vert_{D}\big(Y_{\ell}\big) \big| \\
        &= \big| F^s_i\vert_{D}\big(Y'_{\ell} \big) - F^s_i\vert_{D} \big(X^{(\iota)}_{\ell}\big) \big| \\
        &\le \tThreshold.
    \end{align*}
    The equality follows from the fact that $F^s_i\vert_{D}(X^{(\iota)}_\ell ) = F^s_i\vert_{D}(Y_{\ell})$ due to the value preservation constraints.
    In the last inequality, we use that $Y'_{\ell}$ and $X^{(\iota)}_\ell$ differ in at most~$\tThreshold$ fractional variables.
    By \cref{def:restricted-multilinear-extension}, we have that
    \begin{equation*}
        \big| F^s_i\vert_{D}\big(Y'_{\ell} \big) - F^s_i\vert_{D} \big(X^{(\iota)}_{\ell}\big) \big| = \big| F^s_i \big(Y'_{\ell} \big) - F^s_i \big(Y_{\ell}\big) \big| \le \tThreshold. \qedhere
    \end{equation*}
\end{proof}

After $I$ iterations, we computed an integral $k$-coloring~$\chi$ of small discrepancy such that~$\E[\chi_q] = Y^{(0)}_{q} = 1/k$ for all~$q \in [n]$. 
Since all disjoint sets~$D_j$ are rounded in different iterations (and remain unchanged after), the outcome of the rounding of the sets~$D_1,D_2, \dotsc ,D_{j-1}$ in previous iterations only affects the correlation between elements within $D_j$, not the distribution of single elements in $D_j$ (for the formal argument see the proof of \cref{lem:all-sets-small-set-discrepancy}).
As no two distinct elements of any set~$D_j$ are contained in the same set~$S \in \Ss^b$, the random variables~$\{\chi_u: u\in S\}$ are mutually independent.
Intuitively, it is highly likely that independently rounding~$s = \sAllSetst$ many elements results in each big set in~$\Ss^b$ becoming a $k$-rainbow set.
Hence, we are now in the position to prove \cref{lem:all-sets-small-set-discrepancy}.

\begin{proof}[Proof of \cref{lem:all-sets-small-set-discrepancy}]
    We compute the integral $k$-coloring~$\chi$ in two steps.
    We start by partitioning the $n$ elements into disjoint independent sets in the dependency graph~$G = ([n],E)$ of~$\F^s$ where $E = \set{\set{u,v} \,|\, \exists S \in \Ss^s \cup \Ss^b:  u,v \in S, u \neq v}$.
    Note that the maximal degree~$\delta$ in~$G$ is bounded by~$ts$ due to the size assumption on the sets in~$\Ss^b$.
    Using the classical greedy algorithm, we compute a $\delta + 1$ vertex coloring of~$G$ where each color class is an independent set in~$G$.
    Let~$I \le ts + 1$ and assume \wlogeneral that these independent sets are~$D_1, \dotsc, D_{I}$.
    
    Let~$Y^{(0)} : [n] \rightarrow \Delta_{k}$ denote the fractional $k$-coloring with~$Y^{(0)}_{\ell q} = 1/k$ for all $q \in [n]$ and~$\ell \in [k]$.
    From the choice of~$Y^{(0)}$, we have zero discrepancy
    \begin{equation}
        \label{eq:all-sets-zero-discrepancy-small-sets}
        F^s_i \left(Y^{(0)}_{\ell}\right) = F^s_i \left(Y^{(0)}_{\ell'}\right)
    \end{equation}
    for all~$i \in [m]$ and~$\ell, \ell' \in [k]$.
    Next, we iteratively increase the integral elements in~$Y^{(0)}$ by rounding the fractional elements of one independent set~$D_j$ in each iteration~$j \in [I]$  while all other elements in~$\conj{D_j}$ remain unchanged.
    More precisely, for all~$j \in [I]$, we use \cref{lem:all-sets-rounding-multilinear-extension} to compute a fractional $k$-coloring~$Y^{(j)}$ such that
    \begin{alignat}{3}
        & \forall \bar{d} \in \conj{D_j} \qquad \qquad \qquad  & Y^{(j)}_{\bar{d}} &= Y^{(j-1)}_{\bar{d}}, \label{eq:all-sets-elements-remain-same} \\
        & \forall d \in D_j \qquad \qquad \qquad & \E\big[Y^{(j)}_{d} \big] & = Y^{(j-1)}_{d}, \label{eq:all-sets-expected-value}  \\
        & \forall d \in  D_j \qquad \qquad \qquad \qquad \qquad & Y^{(j)}_{d} \in \set{0,1}^{k} \text{ with } \sum\nolimits_{\ell \in [k]} Y^{(j)}_{d \ell} &= 1, \label{eq:all-sets-integral-coloring} \\
        & \forall i \in [m], \ell \in [k] \qquad \qquad & \left| F^s_i\left(Y^{(j)}_\ell \right) - F^s_i \left(Y^{(j-1)}_{\ell} \right) \right| &\le \tThreshold. \label{eq:all-sets-error-bound-of-restricted-multilinear-extension}
    \end{alignat}
    After~$I$ iterations, we have~$Y^{(I)}_{q} \in \set{0,1}^{k}$ with $\sum_{\ell \in [k]} Y^{(I)}_{q \ell} = 1$ for all~$q \in [n]$ and~$\ell \in [k]$ by \cref{eq:all-sets-integral-coloring}.
    We define the integral $k$-coloring~$\chi$ as follows: for all~$q \in [n]$, set~$\chi_{q} = \ell$ where~$\ell \in [k]$ is the color for which we have~$Y^{(I)}_{q \ell} = 1$.

    In the following, we show that each big set~$S \in \Ss^b$ is with high probability a $k$-rainbow set \wrt~$\chi$.
    We start by showing that the Bernoulli random variables~$\{\chi_u: u\in S\}$ with~$\E[\chi_{u}] = 1/k$ are mutually independent for any~$S \in \Ss^b$.
    Notably, by definition, no two distinct elements in~$D_j$ are contained in the same set~$S \in \Ss$ and all disjoint sets~$D_j$ are rounded in different iterations of our algorithm.
    Together with \cref{eq:all-sets-elements-remain-same} and the choice of~$Y^{(0)}$, it holds for all~$d \in D_j$ that~$d \notin D_1\cup ... \cup D_{j-1}$ and
    \begin{equation}
        \label{eq:all-sets-independent-sets-only-change-in-their-iteration}
        Y_{\ell d}^{(0)} = \dotsc = Y_{\ell d}^{(j-2)} = Y_{\ell d}^{(j-1)} = 1/k.
    \end{equation}
    for all~$\ell \in [k]$.
    Due to \cref{eq:all-sets-expected-value,eq:all-sets-independent-sets-only-change-in-their-iteration}, rounding the elements~$d \in D_j$ in the~$j$-th iteration results in~$\E[Y_{\ell d}^{(j)}] = 1/k$.
    From \cref{eq:all-sets-elements-remain-same}, it follows that they remain unchanged for all iterations~$j+1, \dotsc, I$ and thus~$\E[\chi_{d}] = 1/k$.
    Consequently, conditioning on any outcome of the rounding of elements in~$D_1,D_2, \dotsc,D_{j-1}$, our procedure satisfies that $\chi_d$ is a Bernoulli random variable with~$\E[\chi_{d}] = 1/k$.
    In particular, for some iterations~$j' < j$ and elements~$d'\in D_{j'}$, $d \in D_{j}$,
    the random variable $\chi_{d}$ is distributed the same, regardless on the realization of $\chi_{d'}$ we condition on.
    Therefore, all~$\{\chi_u: u\in S\}$ are independent.
    Finally, with a standard Chernoff bound, we have that each~$S \in \Ss^b$ is with high probability a $k$-rainbow set \wrt~$\chi$ (see \cref{lem:all-sets-independent-proper-coloring-concentration} in Appendix~\ref{sec:appendix}).
    
    It remains to show \cref{eq:all-sets-discrepancy-bound}.
    Combining the definition of~$\chi$ and \cref{eq:all-sets-error-bound-of-restricted-multilinear-extension}, it follows for all~$i \in [m]$ and~$\ell \in [k]$ that
    \begin{align}
        \label{eq:all-sets-error-bound-of-multilinar-extension}
        \left| F^s_i\big(\chi_\ell \big) - F^s_i \left(Y^{(0)}_{\ell} \right) \right|
        &= \left| F^s_i\left(Y^{(I)}_\ell \right) - F^s_i \left(Y^{(0)}_{\ell} \right) \right| \nonumber
        = \sum_{j \in [I]} \left| F^s_i\left(Y^{(j)}_{\ell} \right) - F^s_i \left(Y^{(j-1)}_{\ell} \right) \right| .
    \end{align}
    Due to \cref{eq:all-sets-expected-value} and the fact that~$F^s_i\vert_{D}$ is linear, it holds that~$\big(F^s_i\big(Y^{(j)}_\ell \big)\big)_{j \ge 0}$ is a martingale.
    Applying Azuma's Inequality (see \cref{lem:azumas-inequality}), we have that
    \begin{equation*}
        \Pr\left( \left| F^s_i\big(\chi_\ell \big) - F^s_i \left(Y^{(0)}_{\ell} \right) \right| \geq 6\sqrt{t^3s \ln(ntk)} \right)
         < 2 \exp\big(-2 \ln(ntk)\big)
         < \frac{1}{n^2}
    \end{equation*}
    for all~$\ell \in [k]$.
    Therefore, it holds with high probability that
    \begin{equation}
        \label{eq:all-sets-partial-discrepancy-bound}
        \left| F^s_i\big(\chi_\ell \big) - F^s_i \left(Y^{(0)}_{\ell} \right) \right| <  \OO\big(\tdiscrepancy\big)
    \end{equation}
    for all~$\ell \in [k]$.
    Using the triangle inequality and \cref{eq:all-sets-zero-discrepancy-small-sets,eq:all-sets-partial-discrepancy-bound}, it holds that 
    \begin{align*}
        \Big| F^s_i \big(\chi_\ell \big) - F^s_i \big(\chi_{\ell'} \big) \Big| 
        &\le \left| F^s_i\big(\chi_\ell \big) - F^s_i \left(Y^{(0)}_{\ell} \right) \right| + \left| F^s_i\big(\chi_{\ell'} \big) - F^s_i \left(Y^{(0)}_{\ell'} \right) \right| \\
        &< \OO\big(\tdiscrepancy\big)
    \end{align*}
    for all~$i \in [m]$ and~$\ell, \ell' \in [k]$.
\end{proof}

Putting all the pieces together, we show \cref{thm:all-sets-discrepancy}.

\begin{proof}[Proof of \cref{thm:all-sets-discrepancy}]
    Let~$s = \sAllSetst$.
    Assume \wlogeneral that all sets~$S \in \Ss^b$ have exactly~$s$ elements by arbitrarily picking $s$ elements from~$S$ and removing the rest. 
    Using \cref{lem:all-sets-small-set-discrepancy}, in polynomial time, we compute a $k$-coloring~$\chi : [n] \rightarrow [k]$ such that with high probability each~$S \in \Ss^b$ is a $k$-rainbow set with respect to~$\chi$ and with high probability holds that
    \begin{equation*}
        \label{eq:all-sets-small-set-discrepancy-bound}
        \disc_\chi(\F^s,k) < \OO\big( \tdiscrepancy \big).
    \end{equation*}
    From the definition of $k$-rainbow sets, it follows with high probability that
    \begin{equation*}
        \min\{\One_{S} \cdot \chi_\ell , 1\} 
        \ge \min\{1 , 1\} 
        = 1
    \end{equation*}
    for all~$S \in \Ss^b$ and $\ell \in [k]$.
    As a consequence, we have~$\disc_\chi(\F^b,k) = 0$ with high probability.
    Finally, we can conclude that with high probability holds that
    \begin{equation*}
        \disc_\chi(\F,k) 
        = \disc_\chi(\F^s,k) + \disc_\chi(\F^b,k)
        < \OO\big( \tdiscrepancy \big). \qedhere
    \end{equation*}
\end{proof}

\bibliographystyle{plainurl}
\bibliography{main.bib}

\appendix
\section{Appendix}
\label{sec:appendix}

In this section, we state the previously omitted proofs.

\RestrictedMultilinearExtension*

\begin{proof}
    Let $r = \abs{\Ss}$ and arbitrarily determine an ordering of the sets~$S_1, \dotsc, S_r \in \Ss$.
    For any set~$T \subseteq \set{0,1}^{\abs{D}}$, it holds that
    \begin{equation*}
        f \vert_D (T) = \sum_{j \in [r]} f_{j} \vert_D \big(T \cap (S_j \cap D)\big),
    \end{equation*}
    where all subsets~$S_j \cap D \subseteq [n]$ are disjoint by definition of~$D$.
    Thus, we can decompose any random subset~$R \sim X$ such that~$R(X) = \bigcup_{j \in [r]} R_j(Y_j)$ where all~$R_j$ are independent random subsets, since the decisions to include elements in~$R_j$ are disjoint for all~$j \in [r]$.
    By linearity of expectation and independence of the sets~$R_j$, we have that
    \begin{align*}
        F \vert_D (X)
        &= \E \big[ f \vert_D \big(R(X)\big) \big]
        = \E \left[\sum_{j \in [r]} f_j \vert_D \big(R_j(Y_j)\big) \right] 
        = \sum_{j \in [r]} \E \big[ f_j \vert_D \big(R_j(Y_j)\big) \big] \\
        &= \sum_{j \in [r]} F_j \vert_D (Y_j).
    \end{align*}
    From~$\abs{S_j \cap D} \le 1$ follows that each function~$F_j \vert_D (Y_j)$ is linear.
    Since their sum is linear, the restricted multilinear extension~$F \vert_D (X)$ is linear.
\end{proof}

\RoundingExpectation*

\begin{proof}
    Let $P = \{y\in [0,1]^n : Ay = b\}$.
    Note that every vertex of~$P$ has at most $m$ fractional components.
    Assume otherwise, then the columns of $A$ that correspond to fractional components are linearly dependent. Thus, there exists
    some $\lambda\in \mathbb R^n$ with $A\lambda = 0$ and~$\lambda_i = 0$
    for each $i\in [n]$ with $X_i\in\{0,1\}$.
    Further, for some~$\epsilon > 0$ we have that $X + \epsilon \lambda \in P$ and~$X - \epsilon \lambda \in P$, contradicting the fact that $X$ is a vertex.

    Let $V$ be the set of vertices of $P$. 
    Since $X\in P$, it is a convex combination of vertices in~$V$. 
    Consider the following random process satisfying the properties of the lemma: select each~$v \in V$ with the probability corresponding to its weight in the convex combination of~$X$. It remains to show how to compute such a convex combination equal to $X$ in polynomial time. 
    Formally, we want a solution to the following linear program.
    \begin{alignat*}{3}
    &                & \qquad\quad\qquad \sum\nolimits_{v\in V} \lambda_v v &= X \\
    &                & \qquad\quad\qquad\qquad\quad \sum\nolimits_{v\in V} \lambda_v &= 1 \\
    & \forall v\in V & \qquad\quad\qquad\qquad\qquad \lambda_v &\ge 0 .
    \intertext{As this LP can have exponentially many variables, we construct its dual to solve it.}
        &                  & \qquad\quad\qquad \min \sum\nolimits_{i \in [n]} \mu_i &+ \nu \\
        & \forall v\in V   & \qquad\quad\qquad \sum\nolimits_{i \in [n]} \mu_i v_i + \nu & \ge 0  \\
        & \forall i\in [n] & \qquad\quad\qquad \mu_i & \in \mathbb R  \\
        &                  & \qquad\quad\qquad \nu & \in \mathbb R
    \end{alignat*}
    As zero is a solution to the dual, it is never infeasible and bounded if and only if the primal is feasible.
    Note that the dual has a polynomial number of variables and exponentially many constraints. 
    Using the Ellipsoid method, we can find a solution to it in polynomial time by solving the following separation problem: given a solution $\mu, \nu$, find some $v\in V$ violating the constraint~$\sum_{i \in [n]} \mu_i v_i + \nu < 0$.
    This corresponds to minimizing a linear function (given by~$\mu$) over $V$. Since $P$ is given by a polynomial size linear program, we can compute a vertex~$v\in V$ minimizing a linear function in polynomial time.

    Assuming the primal is feasible, we encounter a polynomial number of constraints attesting that the dual is bounded during the execution of the Ellipsoid method. Restricting ourselves to these constraints and omitting all others results in a bounded and feasible linear program. Consequently, the primal remains feasible when restricted to the corresponding variables. As a result, the restricted primal is of polynomial size and we can compute a solution~$X'$ using any polynomial time LP algorithm.
\end{proof}

\DiscrepancyAnalysis*

\begin{proof}
    We define the proper $k$-coloring~$\chi$ as follows: 
    Each element~$z' \in \conj{Z}$ is already properly colored with~$Y_{z'} \in \set{0,1}^{k}$.
    Let~$\ell' \in [k]$ denote the color for which~$Y_{z'\ell'} = 1$ and set~$\chi_{z'} = \ell'$.
    For all~$z \in Z$, we have~$Y_z \in [0,1]^{k}$ and~$\sum_{\ell \in [k]} Y_{z\ell} = 1$.
    Hence, we randomly color each element~$z$ by choosing~$\chi_z = \ell$ with probability~$Y_{z\ell}$ for each~$\ell \in [k]$.

    In the following, we show that the randomized coloring implies the claimed discrepancy bound.
    Let~$r := \sqrt{(\abs{Z}t^2\cdot (1 + \ln (2mk)))/2}$, then we will show that
    \begin{equation}
        \label{eq:small-sets-probability-of-violated-discrepancy}
        \Pr\big(\abs{f_i(\chi_\ell) - F_i(Y_{\ell})} \ge r \big) < \frac{1}{2mk}
    \end{equation}
    for all $i \in [m]$ and $\ell \in [k]$.
    For now, assume that \cref{eq:small-sets-probability-of-violated-discrepancy} holds.
    Then a standard union bound implies that 
    \begin{equation*}
        \Pr\big(\abs{f_i(\chi_\ell) - F_i(Y_{\ell})} < r \big) > \frac{1}{2}
    \end{equation*}
    for all~$i \in [m]$ and $\ell \in [k]$.
    Using that $F_i(Y_{\ell}) = F_i(Y_{\ell'})$ for all $i \in [m]$ and~$\ell, \ell' \in [k]$ together with the triangle inequality, with probability at least~$1/2$ it holds that
    \begin{equation*}
        \abs{f_i(\chi_\ell)  - f_i(\chi_{\ell'}) } 
        \le \abs{f_i(\chi_\ell) - F_i(Y_{\ell})} + \abs{f_i(\chi_{\ell'}) - F_i(Y_{\ell'})}
        < 2r
    \end{equation*}
    for all $i \in [m]$ and~$\ell, \ell' \in [k]$.
    Thus, we have that $\disc(\F,k) < 2r$.

    As a last step, we prove \cref{eq:small-sets-probability-of-violated-discrepancy}.
    For all~$z \in Z$ and~$\ell \in [k]$, define the random variable~$Y_{z \ell} \in \set{0,1}$ that indicates whether~$\chi_{z} = \ell$ or not. 
    For a fixed color~$\ell \in [k]$, the random variables~$Y_{z \ell}$ are independent for all~$z \in Z$.
    Fixing the values of~$\chi_{z'}$ for all~$z' \in \conj{Z}$, it holds that~$f_i(\chi_\ell) = h_{i\ell}((Y_{z \ell})_{z \in Z})$ for all~$i \in [m]$ and~$\ell \in [k]$.
    Since~$\abs{f_i(\chi_\ell) - f_i(\chi_\ell')} \le t$ whenever two assignments~$\chi_\ell, \chi_\ell' \in \set{0,1}^n$ differ in at most one coordinate, the same applies to the functions~$h_{i\ell}$.
    Finally, we can apply McDiarmid's inequality (see \cref{lem:mcDiarmid-inequality}) such that for all~$i \in [m]$ and~$\ell \in [k]$ holds that
    \begin{align*}
        \Pr\big(\abs{f_i(\chi_\ell) - F_i(Y_{\ell})} \ge r \big) 
        &= \Pr \big(\left|h_{i\ell}((Y_{z \ell})_{z \in Z}) - \E \big[h_{i\ell}((Y_{z \ell})_{z \in Z}) \big] \right| \ge r \big) \\
        &\le 2 \exp \left(- \frac{2r^2}{\abs{Z}t^2}\right) \\
        &= 2 \exp \left(- 1 - \ln (2mk)\right) \\
        &< \frac{1}{2mk}. \qedhere
    \end{align*}
\end{proof}

\begin{lemma}
    \label{lem:all-sets-independent-proper-coloring-concentration}
    Let~$s = \sAllSetst$.
    Let~$\chi : [n] \rightarrow [k]$ be a $k$-coloring such that  for all ~$S \in \Ss^b$ it holds that $\{\chi_{\ell u}: u\in S\}$ are mutually independent random variables with~$\E[\chi_{\ell u}] = 1/k$ for all~$\ell \in [k]$ and $u\in S$.
    With high probability, each~$S \in \Ss^b$ is a $k$-rainbow set \wrt~$\chi$.
\end{lemma}

\begin{proof}
    From the fact that~$\abs{S} = s$ for all~$S \in \Ss^b$, we have that~$\E[\One_{S} \cdot \chi_{\ell}] = s/k$ for all~$\ell \in [k]$.
    Let~$r = s/(2k)$.
    Using a standard Chernoff bound, it holds that
    \begin{align*}
        \Pr \left( r \geq \One_{S} \cdot \chi_\ell \right) 
        = \Pr \big( 1/2 \cdot \E[\One_{S} \cdot \chi_{\ell}] \geq \One_{S} \cdot \chi_\ell \big) 
        &\le \exp\left(-\frac{\sAllSetst}{8k}\right) \\
        &= \exp  \big( - 3 \cdot \ln(ntk) \big) \\
        &< \frac{1}{n^{3} tk}
    \end{align*}
    for all~$\ell \in [k]$ and~$S \in \Ss^b$.
    Using a standard union bound and~$\abs{\Ss^b} \le nt$, we have $r \geq \One_{S} \cdot \chi_\ell$ with probability less than~$1/n^{2}$.
    Consequently, it follows that
    \begin{equation*}
        \One_{S} \cdot \chi_\ell 
        > \frac{s}{2k} 
        \ge 12 \cdot\ln (ntk) 
        \ge 1    
    \end{equation*}
    with probability~$1 - 1/n^{2}$ for all~$\ell \in [k]$ and~$S \in \Ss^b$.
\end{proof}

\end{document}